\documentclass[12pt,reqno,a4paper]{amsart}
\usepackage[top=1.15in, bottom=1.15in, left=1.17in, right=1.17in]{geometry}
\usepackage[T1]{fontenc}
\usepackage[utf8]{inputenc}
\usepackage{newtxtext}
\usepackage{newtxmath}
\usepackage{microtype}
\usepackage{ellipsis}
\usepackage{amssymb}
\usepackage{color,graphicx}
\usepackage{tikz}
\usepackage{xspace}
\usepackage{longtable}
\usepackage[inline]{enumitem}
\usepackage{booktabs}
\usepackage[normalem]{ulem}
\usepackage[english]{babel}
\usepackage[
    colorlinks,
    cite color=blue,
    pdfusetitle
    ]{hyperref}
\usepackage[noabbrev]{cleveref}
\usepackage{amsthm}
\usepackage{mathtools}

\theoremstyle{plain}
\newtheorem{theorem}{Theorem}[section]
\newtheorem{proposition}[theorem]{Proposition}
\newtheorem{lemma}[theorem]{Lemma}
\newtheorem{corollary}[theorem]{Corollary}
\theoremstyle{remark}

\newtheorem{example}[theorem]{Example}

\newtheorem{remark}[theorem]{Remark}

\theoremstyle{definition}
\newtheorem{definition}[theorem]{Definition}

\newcommand{\NN}{{\mathbb N}}
\newcommand{\N}{\NN}

\newcommand{\RR}{{\mathbb R}}
\newcommand{\R}{\RR}
\newcommand{\ZZ}{{\mathbb Z}}
\newcommand{\Z}{\ZZ}

\newcommand{\Dc}{{\mathcal D}}

\newcommand{\Pc}{{\mathcal P}}
\newcommand{\Oc}{{\mathcal O}}
\newcommand{\one}{{\mathbf 1}}

\newcommand{\norm}[1]{\lVert {#1} \rVert}
\newcommand{\llrr}{\norm{\,\cdot\,}}

\DeclareMathOperator{\Newt}{Newt}

\DeclareMathOperator{\conv}{conv}
\DeclareMathOperator{\cone}{cone}

\DeclareMathOperator{\Hom}{Hom}
\DeclareMathOperator{\AP}{AP}

\newcommand{\ie}{i.\,e.\ }

\renewcommand\emptyset\varnothing

\usetikzlibrary{arrows.meta}
\usetikzlibrary{calc}

\tikzset{>/.tip={Stealth}}
\tikzstyle{every picture}=[y={(60:1)}]
\tikzstyle{edge}=[thick]
\tikzstyle{grid}=[circle, fill=gray, inner sep=0, minimum size=2pt]
\tikzstyle{vertex}=[circle, fill, inner sep=0, minimum size=4pt]
\tikzstyle{shaded}=[fill=black!10]
\tikzstyle{boundary}=[lightgray]
\newcommand{\drawgrid}[1]{
  \foreach \x in {0,...,#1}
    \foreach \y in {\x,...,#1}
      \node[grid] at (\x,#1-\y) {};
}

\def\doubleunderline#1{\uuline{#1}}
\def\underline#1{\uline{#1}}

\begin{document}
    \title{Vertex-Maximal Lattice Polytopes Contained in 2-Simplices}

    \author[J.-P. Litza]{Jan-Philipp Litza}
    \author[C. Pegel]{Christoph Pegel}
    \author[K. Schmitz]{Kirsten Schmitz}

    \address{Department of Mathematics, University of Bremen, Germany}
    \email{jplitza@math.uni-bremen.de}
    \address{Institute for Algebra, Number Theory, and Discrete Mathematics, Leibniz University Hannover, Germany}
    \email{pegel@math.uni-hannover.de}
    \address{Department of Mathematics, University of Bremen, Germany}
    \email{kschmitz@math.uni-bremen.de}

    \begin{abstract}
Motivated by the problem of bounding the number of rays of plane tropical curves we study the following question: Given $n\in\N$ and a unimodular $2$-simplex $\Delta$ what is the maximal number of vertices a lattice polytope contained in $n\cdot \Delta$ can have? We determine this number for an infinite subset of $\N$ by providing a family of vertex-maximal polytopes and give bounds for the other cases.
\end{abstract}

    \maketitle
    
    \section{Introduction} \label{sec-introduction}

In \cite{JS} the authors study upper bounds for face numbers of tropical varieties of given degree and dimension. To our surprise, apart from their results very little seems to be known even about basic cases like plane tropical curves. In this note, we deal with the question of how many vertices a lattice polytope contained in a dilated unimodular $2$-simplex can have, thus providing bounds for the number of rays of plane tropical curves of a given degree.

Let $\one$ denote the vector $(1,1,1)$ in the lattice $\Z^3$ and denote by $\Lambda$ the lattice $\Z^3/\Z\one$.
Let $V=\Lambda\otimes_{\Z}\R$, which we will consider as a two-dimensional $\R$-vector space with lattice $\Lambda$. We want to study non-empty lattice polytopes in $V$ up to translation. Thus, we consider the equivalence class $[P]$ of a lattice polytope $P$ in $V$ with respect to the natural action of $\Lambda$ on $V$. We will denote the set of all these equivalence classes of non-empty lattice polytopes in $V$ by $\Pc$. When unambiguous, we will refer to an equivalence class of polytopes $[P]$ just as a polytope.

For two polytopes $[P], [Q]\in\Pc$ we will write $[P]\subset [Q]$ if there are representatives $P'\in[P]$ and $Q'\in[Q]$ such that $P'\subset Q'$. Furthermore, we define the Minkowski sum on $\Pc$ by $[P]+[Q]=[P+Q]$. Note that this is well defined, since a translation of a representative of one of the summands results in a translation of the sum. In particular, for $n\in\N$ let $n\cdot[P]=[n\cdot P]$ be the Minkowski sum of $n$ copies of $[P]$.

One polytope of particular importance will be the unimodular simplex $\Delta=\conv(a_1,a_2,a_3)$ with $a_1=\overline{0}$, $a_2=\overline{e_1}$ and $a_3=\overline{e_1}+\overline{e_2}$ as depicted in the left hand side of \Cref{fig:delta}.
The main question we are addressing in this article is the following: Given $n\in\mathbb N$, what is the maximal number of vertices a polytope $[P]\subset n\cdot[\Delta]$ can possibly have? We will refer to this number as $A(n)$ and call a polytope $[P]\subset n\cdot[\Delta]$ with $A(n)$ vertices \emph{$n$-vertex-maximal} or just \emph{vertex-maximal}. On the right hand side of \Cref{fig:delta}, we see a polytope $[P]$ with $6$ vertices contained in $4\cdot[\Delta]$. This polytope is indeed $4$-vertex-maximal, so $A(4)=6$.

In \Cref{sec-polytopeval}, we study the map $n\colon\Pc\to\N$ that assigns to a polytope $[P]$ the minimal $n\in\mathbb N$ such that $[P]\subseteq n\cdot[\Delta]$. We refer to $n([P])$ as the \emph{simplicial diameter} of $[P]$ and emphasize that this number depends on the choice of the simplex $[\Delta]$. Using this we can describe $A(n)$ as the largest $A\in\N$ such that there is a polytope $[P]$ with $n([P])\le n$ and $f_0([P])=A$, where $f_0$ denotes the number of vertices. We find that the simplicial diameter is a Minkowski-additive valuation and give an alternative description in terms of the edge-defining linear forms of $[P]$. Using these results we construct an infinite family of vertex-maximal polytopes in \Cref{sec-infinitefamily}. For $n\in\N$ not covered in \Cref{sec-infinitefamily} we determine bounds for $A(n)$ and examine additional properties of vertex-maximal polytopes in \Cref{sec-intermediate}. The application to tropical geometry is explained in \Cref{sec-related}, along with the recovery of the asymptotic behavior of $A(n)$ that has been studied in \cite{BP}.

\begin{figure}
    \begin{tikzpicture}[scale=1]
        \path[shaded,shift={(4,-2)}] (0,3) -- (2,2) -- (3,1) -- (2,0) -- (1,0) -- (0,2) -- cycle;
        \foreach \y in {-1,0,1}
            \foreach \x in {-2,...,8}
                \node[grid] at (\x-\y,2*\y) {};
        \foreach \y in {0,1}
            \foreach \x in {-2,...,9}
                \node[grid] at (\x-\y,2*\y-1) {};
        \draw[shaded]
            (0,0) node[left] {$a_1$}
            -- (1,0) node[right] {$a_2$}
            -- (0,1) node[above] {$a_3$}
            -- cycle;
        \draw[thick,->] (0,0) -- node[below,inner sep=1pt] (e1) {\tiny$\overline{e_1\!}$} (1,0);
        \draw[thick,->] (0,0) -- node[above right,inner sep=0pt] {\tiny$\overline{e_2\!}$} (-1,1);
        \draw[thick,->] (0,0) -- node[below right,inner sep=0pt] {\tiny$\overline{e_3\!}$} (0,-1);
        \node at (1/3,1/3) {$\Delta$};

        \begin{scope}[shift={(4,-2)}]
            \draw (0,0) -- (4,0) -- (0,4) -- cycle;
            \draw[edge] (0,3) -- (2,2) -- (3,1) -- (2,0) -- (1,0) -- (0,2) -- cycle;
        \end{scope}
    \end{tikzpicture}
    \caption{The unimodular simplex $\Delta$ and the projected standard basis $(\overline{e_1},\overline{e_2},\overline{e_3})$ in $V$ on the left. On the right is a lattice polytope $P$ with 6~vertices in a translated copy of $\Delta$, scaled by a factor of $n([P])=4$.}
    \label{fig:delta}
\end{figure}
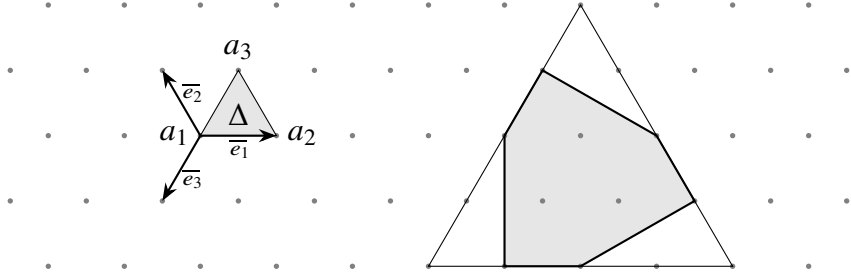
    \section{Simplicial Diameter as a Valuation} \label{sec-polytopeval}

To get a better understanding of the map $n\colon \Pc\to\N$, we use a correspondence of polytopes up to translation in $\Pc$ and certain vector configurations in the dual space $V^*$ that goes back to Minkowski \cite{Minkowski1897,Minkowski1903} and is detailed in \cite[Chapter~8]{Schneider}. This allows us to describe the simplicial diameter of a polytope in terms of its corresponding vector configuration, considerably simplifying the search for vertex-maximal polytopes.

\medskip

Letting $\Lambda^*=\Hom_{\Z}(\Lambda,\Z)$ be the dual lattice of $\Lambda$ we have a canonical isomorphism $V^*=\Hom_{\R}(V,\R)\cong \Lambda^*\otimes_{\Z} \R$.
Let $P\subset V$ be a lattice polytope. Using the notation as in \cite{GS}, for $v\in V^*\setminus\{0\}$ denote by $S(P,v)$ the face of $P$ defined by $v$, that is,
\begin{equation*}
    S(P,v) = \bigg\{\,x \in P : v(x)=\max_{y\in P} v(y) \,\bigg\}.
\end{equation*}
Note that for $[Q]=[P]$ we have $[S(P,v)]=[S(Q,v)]$ and hence $S([P],v)\coloneqq[S(P,v)]$ is well-defined.

We say a vector $x\in\Lambda$ is \emph{primitive} if whenever $x=\lambda y$ for $\lambda\in\N$ and $y\in\Lambda$, we have $\lambda=1$. The same applies to vectors in the dual lattice. Another notion we use is the \emph{lattice length} of elements of $\Lambda$ and $\Lambda^*$, as well as lattice length of line segments $[L]\in\Pc$. Any vector $x\in \Lambda\setminus\{0\}$ can be uniquely expressed as $x=\lambda\hat x$ for $n\in\N$ and a primitive vector $\hat x\in\Lambda$. In this situation we define the lattice length of $x$ as $\ell(x)=\lambda$ and furthermore $\ell(0)=0$. The lattice lengths of elements of $\Lambda^*$ is defined in the same way and for a line segment $[L]\in\Pc$ with $L=\conv(a,b)$ we let $\ell([L])=\ell(b-a)$.

We call a finite set $T=\{v_1,\dots,v_t\}$ of vectors in $V^*\setminus\{0\}$ a \emph{vector configuration} if $v_i\notin\cone(v_j)$ for $i\neq j$. That is, there are no two vectors in the same direction. The vector configuration is said to be \emph{balanced} if $\sum_{v\in T} v = 0$ and we denote by $\Dc$ the set of all balanced lattice vector configurations in $\Lambda^*\setminus\{0\}$.

For a polytope $[P]\in \Pc$ let $D([P])$ be the vector configuration consisting of all $v \in \Lambda^*\setminus\{0\}$ defining an edge $S([P],v)$ of lattice length $\ell(v)$ as illustrated in \Cref{fig:correspondence}. Note that $|D([P])|$ is equal to the number of vertices of $[P]$, denoted by $f_0([P])$, except for the point $[*]$, having one vertex but $|D([*])|=0$.

\begin{figure}
    \begin{tikzpicture}[scale=.7]
        \begin{scope}[xshift=-5cm]
            \path[shaded] (3,1) -- (3,0) -- (2,0) -- (1,1) -- (0,3) -- (0,4) -- (1,4) -- (2,3);
            \foreach \y in {0,1,2}
                \foreach \x in {1,...,4}
                    \node[grid] at (\x-\y,2*\y) {};
            \foreach \y in {0,1,2,3}
                \foreach \x in {2,...,4}
                    \node[grid] at (\x-\y,2*\y-1) {};
            \draw[edge] (3,1) -- (3,0) -- (2,0) -- (1,1) -- (0,3) -- (0,4) -- (1,4) -- (2,3);
            \draw[edge,green!80!black] (2,3) -- (3,1);
        \end{scope}
        \draw[|-{Computer Modern Rightarrow}] (-1,2) -- node[pos=.5, above] {$D$} (1,2);
        \begin{scope}[shift={(4,2)},rotate=90]
            \foreach \y in {-1,0,1}
                \foreach \x in {-2,...,2}
                    \node[grid] at (\x-\y,2*\y) {};
            \foreach \y in {0,1}
                \foreach \x in {-1,...,2}
                    \node[grid] at (\x-\y,2*\y-1) {};
            \begin{scope}[every path/.style={edge,->}]
                \draw[green!80!black] (0,0) -- (1,-2);
                \draw (0,0) -- (1,0);
                \draw (0,0) -- (1,-1);
                \draw (0,0) -- (0,-1);
                \draw (0,0) -- (-1,0);
                \draw (0,0) -- (-1,1);
                \draw (0,0) -- (-1,2);
                \draw (0,0) -- (0,1);
            \end{scope}
        \end{scope}
    \end{tikzpicture}
    \caption{A polytope $[P]\in\Pc$ on the left and the corresponding vector configuration $D([P])\in\Dc$ on the right. One edge and its corresponding outer normal vector are highlighted in green.}
    \label{fig:correspondence}
\end{figure}
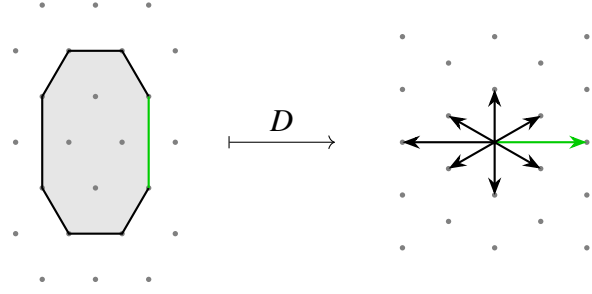

\begin{proposition}
    There is a bijection $D\colon \Pc \to \Dc$ given by $[P]\mapsto D([P])$.
\end{proposition}

\begin{proof}
    Note that the dimensions of $[P]$ and of the linear hull of $D([P])$ are equal, so for each dimension $d=0,1,2$ we show that translation classes of $d$-dimensional polytopes in $\Pc$ are in bijection with vector configurations in $\Dc$ spanning a subspace of dimension $d$. For $d=0$, the point $[*]$ is the only $0$-dimensional polytope in $\Pc$ and $D([*])=\emptyset$ is the only configuration in $\Dc$ with $0$-dimensional span. The configurations in $\Dc$ with $1$-dimensional span are exactly $\{\lambda \hat{v},-\lambda \hat{v}\}$ for primitive $\hat{v}\in\Lambda^*$ and $\lambda\in \N\backslash\{0\}$. A configuration $\{\lambda \hat{v},-\lambda \hat{v}\}\in\Dc$ corresponds to a line segment $[L]$ of lattice length $\lambda$ with $L\subset \ker(v)$. 

    For $d=2$ this is the $2$-dimensional case of Minkowski's correspondence between full-dimensional polytopes up to translation and their facet normal directions together with facet volumes as stated in \cite[Theorem~8.1.1, Thm.~8.2.1]{Schneider}.
\end{proof}

We will equip the real vector space $V^*$ with an asymmetrical norm, referring to \cite[Section~2.2.2]{COB} for details. Recall that $a_1$, $a_2$, $a_3$ are the vertices of the unimodular simplex $\Delta\subset V$. For pairwise distinct $i,j,k\in\{1,2,3\}$ let $b_i\in V^*$ be the unique primitive linear form such that $S(\Delta, b_i) = \conv(a_j, a_k)$. Note that we have $\{b_1,b_2,b_3\}=D([\Delta])$.
In $V^*$ consider the lattice polytope $K=\conv(b_1,b_2,b_3)$.
Since $K$ contains $0$ in its interior, it is an absorbing subset of $V^*$, meaning that for each $v\in V^*$ there is some $\lambda\ge 0$ such that $v\in\lambda\cdot K$. Thus, we obtain a Minkowski functional
\begin{align*}
    \llrr\colon V^* & \longrightarrow \R_{\ge0},\\
    v & \longmapsto \inf\left\{\,\lambda\in\R_{\ge0} : v\in \lambda\cdot K\,\right\}.
\end{align*}

As $K$ is also convex, the map $\llrr$ is an asymmetric norm, \ie it is positive definite, positive homogeneous and satisfies the triangle inequality.
In \Cref{fig:norm} we illustrate the integral level sets of $\llrr$ together with the face fan of $K$, which is a complete fan with three $2$-dimensional cones $C_i=\cone(b_j,b_k)$ for pairwise distinct $i,j,k\in\{1,2,3\}$.

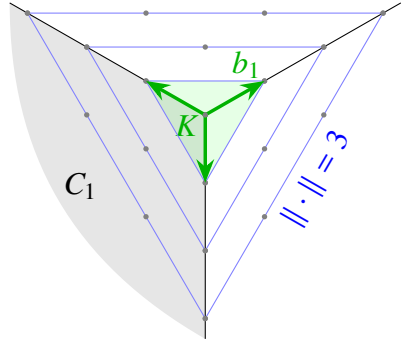
\begin{figure}
    \begin{tikzpicture}[rotate=90,scale=.9]
        \newcommand{\gridlimit}{3}
        \path[shaded] (0,0) -- (-\gridlimit-0.3,0) to[bend left] (0,\gridlimit+0.3) -- cycle;
        \node at (-2.1,2.1) {$C_1$};
        
        \draw
            (0,0) -- (0,\gridlimit+0.3)
            (0,0) -- (\gridlimit+0.3,-\gridlimit-0.3)
            (0,0) -- (-\gridlimit-0.3,0);
            
        \path[fill=green,opacity=.1] (0,1) -- (-1,0) -- (1,-1) -- cycle;
        \node[green!70!black] at (-.3,.3) {$K$};
        
        \foreach \x in {0,...,\gridlimit}
            \draw[blue!50] (0,\x) -- (-\x,0) -- (\x,-\x) -- cycle;
        \tikzstyle{primitive}=[fill=red,vertex]
        \node[blue,rotate=60,anchor=north] at (0,-1.5) {$\llrr = 3$};

        \foreach \x in {0,...,\gridlimit}
            \foreach \y in {\x,...,\gridlimit} {
                \pgfmathsetmacro{\gridstyle}{gcd(\x, \gridlimit-\y) == 1 ? "primitive" : ""}
                \tikzstyle{primgrid}=[grid,\gridstyle]
                \node[grid,x={(-1,0)}] at (\x,\gridlimit-\y) {};
                \node[grid,x={(-1,0)},y={(-1,-1)}] at (\x,\gridlimit-\y) {};
                \node[grid,x={(1,-1)}] at (\x,\gridlimit-\y) {};
            }
        
        \begin{scope}[->,green!70!black,line width=1.3pt]
            \draw (0,0) -- (0,1);
            \draw (0,0) -- (-1,0);
            \draw (0,0) -- (1,-1);
        \end{scope}
        \node[green!70!black,above left,inner sep=1pt] at (1,-1) {$b_1$};
        
        \node[grid] at (0,0) {};
    \end{tikzpicture}
    \caption{The vectors \(b_1\), \(b_2\) and \(b_3\) spanning the absorbing set $K$ together with its face fan and the integral level sets of $\llrr$.}
    \label{fig:norm}
\end{figure}

\begin{lemma} \label{lem:norm-linear}
    On each cone $C_i=\cone(b_j, b_k)$, we have $\llrr=b_j^*+b_k^*$, where $b_j^*$, $b_k^*$ is the dual basis of $b_j$, $b_k$. In particular $\llrr$ is linear on each cone.
\end{lemma}

\begin{proof}
    Without loss of generality we can consider $C_1=\cone(b_2,b_3)$. For $v\in C_1$ and $\lambda\in\R_{\geq0}$ we have $v\in \lambda K$ if and only if $v\in\conv(0,\lambda b_2,\lambda b_3)$ if and only if $v=c_2b_2+c_3b_3$ for some $c_2,c_3\geq0$ with $c_2+c_3\leq \lambda$. Hence, $\norm{ v}=\lambda$ if and only if $v=c_2b_2+c_3b_3$ with $c_2+c_3= \lambda$.
    Thus, letting $b_2^*$, $b_3^*$ be the dual basis of $b_2$, $b_3$, we have $\norm{v}=(b_2^*+b_3^*)(v)$ for $v\in C_1$.
\end{proof}

Using this asymmetric norm, we can associate another invariant to polytopes in $\Pc$, which will turn out to coincide with the simplicial diameter:
\begin{equation*}
\begin{split}
    m\colon \Pc &\longrightarrow \R_{\ge 0},\\
    [P] &\longmapsto \tfrac13\smashoperator{\sum_{v\in D([P])}} \norm{ v }.
\end{split}
\end{equation*}

We want to show that the map $m$ is Minkowski-additive, \ie it satisfies $m([P]+[Q])=m([P])+m([Q])$. This will then imply that it is a valuation which allows us to show that it agrees with the simplicial diameter. We start by translating Minkowski sums of polytopes to an operation on balanced vector configurations.

\begin{lemma}
    \label{lem:corres_mink_sum}
    Given two polytopes $[P]$, $[Q]$, we have
    \begin{equation*}
        D([P]+[Q]) = D([P]) \boxplus D([Q]),
    \end{equation*}
    where $T\boxplus U$ denotes the union $T\cup U$ with vectors sharing a direction replaced by their sum:
    \begin{align*}
        T \boxplus U
        &= \left\{\, v + w : v\in T,\ w\in U,\ \cone(v)=\cone(w)\,\right\}\\
        &\qquad\cup \left\{\,v \in T : \cone(v)\cap U=\emptyset\,\right\}\\
        &\qquad\cup \left\{\,w \in U : \cone(w)\cap T=\emptyset\,\right\}.
    \end{align*}
\end{lemma}

\begin{proof}
    For the inclusion $D([P]+[Q])\subset D([P])\boxplus D([Q])$, let $v\in D([P]+[Q])$ so that $v=\lambda \hat v$ for an edge $S([P]+[Q],\hat v)$ of lattice length $\lambda$. Note that $S([P]+[Q],\hat v)$ decomposes as the Minkowski sum 
    \begin{equation} \label{eq:S-minkowski}
        S([P]+[Q], \hat v) = S([P], \hat v) + S([Q], \hat v)
    \end{equation}
    as stated in \cite[Lemma~2.1.4]{GS}. 
    Since $S([P]+[Q], \hat v)$ is an edge, we are in one of the following situations: Either the two summands are parallel edges or one summand is an edge and the other is a vertex. If the summands are parallel edges, their lattice lengths add up so that $v=\lambda\hat v = \mu\hat v+\nu\hat v$ for $\mu\hat v\in D([P])$ and $\nu\hat v\in D([Q])$. If, without loss of generality, $S([P],\hat v)$ is an edge and $S([Q],\hat v)$ is a vertex , we have $\ell(S([P+Q],\hat v)) = \ell(S([P],\hat v))$ so that $v\in D([P])$ and furthermore $\cone(v)\cap D([Q])=\emptyset$, since for any $w\in\cone(v)$ we have $S([Q],w)=S([Q],\hat v)$, which is not an edge. In both cases, we conclude $v\in D([P])\boxplus D([Q])$.

    For the other inclusion, first consider $v+w\in D([P])\boxplus D([Q])$ where $v\in D([P])$, $w\in D([Q])$ and $\cone(v)=\cone(w)$. Then $v=\mu\hat v$ and $w=\nu\hat v$ for a primitive $\hat v$ and $\mu,\nu\in\N$. Hence, $S([P], \hat v)$ and $S([Q],\hat v)$ are parallel edges of lattice lengths $\mu$ and $\nu$, respectively. Their Minkowski sum is the edge $S([P]+[Q],\hat v)$ of $[P]+[Q]$ with lattice length $\lambda=\mu+\nu$ by \eqref{eq:S-minkowski}. Thus, we have $v+w=\lambda\hat v\in D([P]+[Q])$. At last, consider $v\in D([P])\boxplus D([Q])$ such that---without loss of generality---$v\in D([P])$ and $\cone(v)\cap D([Q])=\emptyset$. Writing $v=\lambda\hat v$, we see that $S([Q],\hat v)$ is a vertex and $S([P],\hat v)$ an edge of lattice length $\lambda$. Using \eqref{eq:S-minkowski}, we conclude that $S([P]+[Q],\hat v)$ is an edge of lattice length $\lambda$ as well so that $v\in D([P]+[Q])$.
\end{proof}

\begin{lemma}
    The map $m$ is Minkowski-additive, that is, for polytopes $[P]$ and $[Q]$ we have $m([P+Q]) = m([P])+m([Q])$. In particular, $m$ is a polytope valuation.
\end{lemma}

\begin{proof}
    The additivity follows immediately from \Cref{lem:corres_mink_sum} and the positive homogeneity of $\llrr$. Let $P$ and $Q$ be polytopes in the usual sense, such that $P\cup Q$ is again a polytope. From the identity
    \begin{equation*}
        P + Q = P \cup Q + P \cap Q
    \end{equation*}
    due to \cite{Sallee} together with Minkowski-additivity, we obtain
    \[
    m([P])+m([Q]) = m([P+Q]) = m([P\cup Q+P\cap Q]) =m([P\cup Q])+m([P\cap Q]).
    \]
    Hence, any Minkowski-additive map on $\Pc$, in particular $m$, is in fact a polytope valuation.
\end{proof}

\begin{theorem}\label{thm:sum-equiv}
    For every polytope $[P]\in\Pc$, we have $n([P])=m([P])$.
\end{theorem}

\begin{proof}
    Let $n=n([P])$ and choose a representative $P$ of $[P]$ that is contained in $n\cdot\Delta$. We subdivide the polytope $P$ and the encompassing simplex $n\cdot\Delta$ as indicated in \Cref{fig:simplex_subdiv}:
    For pairwise distinct $i,j,k\in\{1,2,3\}$, choose a lattice point $x_i\in P$ on the edge $n\cdot \conv(a_j,a_k)$ of $n\cdot \Delta$. Such a point exists by minimality of $n$ such that $[P]\subset n\cdot[\Delta]$. We then let
        $Q_0 = \conv\{x_1,x_2,x_3\}$,
        $S_i = \conv\{n\,a_i,x_j,x_k\}$ and
        $Q_i = S_i \cap P$.
    This way, $Q_0$, $S_1$, $S_2$ and $S_3$ form a subdivision of $n\cdot\Delta$. Intersecting the cells with $P$, we obtain a subdivision of $P$ into $Q_0$, $Q_1$, $Q_2$ and $Q_3$.

    \begin{figure}
        \begin{tikzpicture}[scale=.5]
            \draw[red]
                (0,6) -- (0,0) node[below left, black] {$n\,a_1$} -- (5,0)
                (0,6) -- (0,10) node[above, black] {$n\,a_3$} -- (4,6)
                (4,6) -- (10,0) node[below right, black] {$n\,a_2$} -- (5,0);
            \draw[blue]
                (0,6) -- (1,3) -- (3,1) -- (5,0)
                (5,0) -- (6,0) -- (7,1) -- (7,2) -- (5,5) -- (4,6)
                (4,6) -- (2,7) -- (1,7) -- (0,6);
            \draw[blue,dash pattern=on 3pt off 7pt,postaction={draw,red,dash pattern=on 3pt off 7pt,dash phase=5pt}]
                (0,6) node[above left,black] {$x_2$}
                -- (5,0) node[below,black] {$x_3$}
                -- (4,6) node[above right,black] {$x_1$}
                -- cycle;

            \node[red] at (1,1) {$S_1$};
            \node[red] at (8,1) {$S_2$};
            \node[red] at (1,8) {$S_3$};
            \node at (3,4) {$Q_0$};
            \node[blue] at (2.4,2.4) {$Q_1$};
            \node[blue] at (6,2) {$Q_2$};
            \node[blue] at (1.75,6.5) {$Q_3$};
        \end{tikzpicture}
        \caption{Subdivision of the polytope $P=Q_0\cup \textcolor{blue}{Q_1\cup Q_2 \cup Q_3}$ and the
        encompassing simplex $n\cdot\Delta = Q_0\cup \textcolor{red}{S_1\cup S_2\cup S_3}$ as used in the proof of \Cref{thm:sum-equiv}.}
        \label{fig:simplex_subdiv}
    \end{figure}
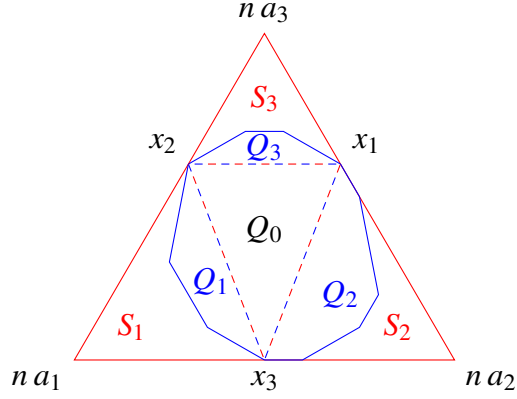

    Using the valuation property, we have
    \begin{equation}
        m([P]) = m([Q_1]) + m([Q_0\cup Q_2\cup Q_3]) - m([Q_1\cap(Q_0\cup Q_2\cup Q_3)]).
        \label{eq:polytope_subdiv}
    \end{equation}
    We claim that $m([Q_1])=m([S_1])$. If $x_2=x_3=n\,a_1$, we have $S_1=Q_1=\{n\, a_1\}$ and the claim is trivial. If $x_2\neq x_3$, the polytopes $Q_1$ and $S_1$ share the edge $\conv(x_2,x_3)$ and its corresponding outer normal linear form $w\in\Lambda^*$ appears in $D([Q_1])$ as well as $D([S_1])$. Let $T=D([Q_1])\setminus\{w\}$ and $U=D([S_1])\setminus\{w\}$. From the balancing condition we obtain
    \[
        \sum_{v\in T} v
        \ = \ -w
        \ = \sum_{v\in U} v.
    \]
    One easily verifies $T,U\subset C_1$, so we can use the linearity shown in \Cref{lem:norm-linear} to obtain
    \begin{align*}
        3\, m([Q_1])
        &= \norm{ w } + \sum_{v\in T} \norm{v}
        = \norm{ w } + \raisebox{-.7ex}{\Bigg\lVert} \sum_{v\in T} v\ \raisebox{-.7ex}{\Bigg\rVert}
        \\&= \norm{ w } + \raisebox{-.7ex}{\Bigg\lVert} \sum_{v\in U} v\ \raisebox{-.7ex}{\Bigg\rVert}
        = \norm{ w } + \sum_{v\in U} \norm{v}
        = 3\, m([S_1]).
    \end{align*}
    Continuing \cref{eq:polytope_subdiv} using this knowledge, we see that
    \[
        m([P])
        = m([S_1]) + m([Q_0\cup Q_2\cup Q_3]) - m([S_1\cap(Q_0\cup Q_2\cup Q_3])
        = m([Q_0\cup S_1\cup Q_2\cup Q_3]).
    \]
    In the same way, we can replace $Q_2$ and $Q_3$ by $S_2$ and $S_3$, respectively, to arrive at
    $m([P])=m([n\cdot\Delta])$, which we calculate to be $n([P])$:
    \begin{equation*}
        m([P]) = m([n\cdot\Delta]) = n\cdot m([\Delta]) = \tfrac{n}{3} \big(\norm{ b_1 } + \norm{ b_2 } + \norm{ b_3 }\big) = n. \qedhere
    \end{equation*}
\end{proof}

\begin{remark}
    While the result of \Cref{thm:sum-equiv} is the theoretically more useful, an easier computation of the simplicial diameter $n([P])$ can be given by evaluating the three outer normal linear forms $b_i$ on the polytope $P$ and adding their maxima:
    \begin{equation} \label{eq:sumofmax}
        n([P])
        = \max_{x\in P} b_1(x) + \max_{x\in P} b_2(x) + \max_{x\in P} b_3(x).
    \end{equation}
\end{remark}

\begin{remark}
    If $V$ is any real vector space of dimension $d$ and $\Delta\subset V$ a $d$-simplex, the map that assigns to a polytope $P\subset V$ the smallest $r\in\R_{\ge 0}$ such that a translate of $P$ is contained in $r\cdot\Delta$ can be expressed analogous to \eqref{eq:sumofmax} as a sum of $d+1$ maximized linear forms. Hence, this more general simplicial diameter will always be a Minkowski-additive, translation invariant valuation on polytopes in $V$.
    
    In the case $d=2$ we have $V=\R^2$ without loss of generality. Replacing lattice length by euclidean length and primitive vectors by unit vectors, we still have a correspondence of translation classes of polytopes in $V$ and balanced vector configurations in $V^*$. In this setting, $D([P])$ consists of vectors $v\in V^*$ such that $S([P],v)$ is an edge of same euclidean length as $v$. Here the euclidean norm on $V^*$ is defined such that $\langle x, \,\cdot\,\rangle$ has the same norm as $x$, where $\langle\,\cdot\,,\,\cdot\,\rangle$ is the standard inner product on $\R^2$. We can then set $K=\conv(D([\Delta]))$ as before and define $m\colon\Pc\to\R_{\ge 0}$ using the Minkowski-functional given by $K$. The same methods as above still yield $m([P])=r$ even in this general setting.
\end{remark}

    \section{An Infinite Family of Vertex-Maximal Polytopes} \label{sec-infinitefamily}

In this section, we construct an infinite family of vertex-maximal polytopes. Recall that a polytope $[P]$ is vertex-maximal if there is some $n\in\N$ such that $[P]$ has the largest number of vertices among all polytopes contained in $n\cdot[\Delta]$, or equivalently, among all polytopes contained in $n([P])\cdot[\Delta]$. The main tool for this construction will be a notion of \emph{saturated} sets in $\Lambda^*$.

\begin{definition}
    A set $S\subset \Lambda^*$ is said to be \emph{saturated} if it satisfies that
    \begin{enumerate}
        \item every $v\in S$ is primitive, and
        \item whenever $w\in\Lambda^*$ is primitive and $\norm{w} < \norm{v}$ for some $v\in S$, then $w\in S$.
    \end{enumerate}
\end{definition}

In other words, for some given norm bound $q\in\N$, the set $S$ contains all primitive elements $v\in\Lambda^*$ with $\norm{v} \le q$ and possibly some---but not all---with $\norm{v} = q+1$. Hence, letting $S_{\le q}$ denote the balanced saturated set of all primitive vectors $v\in\Lambda^*$ with $\norm{v}\le q$, every saturated set $S$ is of the form $S = S_{\le q} \cup R$ for unique $q\in\N$ and $R\subsetneqq S_{\le q+1}\setminus S_{\le q}$. On each cone $C_i=\cone(b_j,b_k)\subset V^*$ our asymmetric norm may be described as the $1$-norm with respect to the basis $b_j$, $b_k$ as shown in \Cref{lem:norm-linear}. Thus, denoting by $\widetilde C_i$ the half-open cone in $C_i$ with strictly positive $b_j$-coordinate, the number of primitive vectors in $\widetilde C_i$ of norm $l$ is equal to $\varphi(l)$, where $\varphi$ is Euler's totient function. Hence, the number of primitive vectors in $\Lambda^*$ of norm $l$ is equal to $3\,\varphi(l)$ and for any saturated set $S$ we obtain the norm bound as the largest $q\in\N$ such that $3\sum_{l=1}^q \varphi(l)\le|S|$.

Since all vectors in a saturated set are primitive, it is always a vector configuration. Hence, if a saturated set $S\subset\Lambda^*$ is balanced, there is a unique polytope $[P_S]$ with $D([P_S])=S$. As we will show now, these polytopes are always vertex-maximal.

\begin{lemma} \label{lem:saturated}
    Let $k\in\N$ and for each vector configuration $T\subset\Lambda^*$ with $|T|=k$ consider the sum $\sum_{v\in T} \norm{v}$. A vector configuration minimizes this sum if and only if it is saturated.
\end{lemma}

\begin{proof}
    Let $T\subset\Lambda^*$ be a vector configuration with $|T|=k$ such that $\sum_{v\in T}\norm{v}$ is minimal. If any of the $v\in T$ were non-primitive, they could be exchanged for their primitive counterparts $v/\ell(v)$. Since $\llrr$ is positive homogeneous, this would reduce the sum of norms, contradicting the minimality. If there were some primitive $v\in T$, $w\notin T$ with $\norm{w} < \norm{v}$, we could replace $v$ by $w$ and lower the sum of norms, again contradicting the minimality. Hence, $T$ is saturated.
    
    Now let $S\subset\Lambda^*$ be any saturated set with $|S|=k$. Since the norm bound $q$ such that $S=S_{\le q}\cup R$ is determined by $|S|=k$, the considered sum
    \begin{equation*}
        \sum_{v\in S} \norm{v} = 3 \sum_{l=1}^q l\,\varphi(l) + (q+1)\,|R|
    \end{equation*}
    is also determined by $k$. Hence, it is minimal for all saturated sets with $|S|=k$.
\end{proof}

\begin{corollary} \label{cor:saturated-maximal}
    Every polytope $[P]$ with saturated $D([P])$ is vertex-maximal and furthermore minimizes the simplicial diameter $n([P])$ under all polytopes with the same number of vertices.
\end{corollary}

\begin{proof}
    Let $[P]$ be a polytope such that $D([P])$ is saturated. Assume that $[P]$ is not vertex-maximal, so there is a polytope $[Q]$ with $n([Q])\le n([P])$ and $f_0([Q]) > f_0([P])$. Note that this is equivalent to $|D([Q])|>|D([P])|$ and let $D([Q]) = \{w_1,w_2,\dots,w_{|D([Q])|}\}$. Using \Cref{thm:sum-equiv} and that all $\norm{w_i}>0$, we obtain the inequality
    \begin{equation*}
        \smashoperator{\sum_{i=1}^{|D([P])|}} \norm{w_i} \ < \ 
        \smashoperator{\sum_{i=1}^{|D([Q])|}} \norm{w_i} \ = \
        3\,n([Q]) \ \le \ 
        3\,n([P]) \ = \ 
        \smashoperator{\sum_{v\in D([P])}} \norm{v}.
    \end{equation*}
    This is in contradiction to \Cref{lem:saturated} and we conclude that $[P]$ was vertex-maximal.

    The fact that $[P]$ minimizes $n([Q])$ under all polytopes $[Q]$ with the same number of vertices is a direct consequence of \Cref{lem:saturated} together with \Cref{thm:sum-equiv}.
\end{proof}

For the vertex-maximal polytopes obtained from balanced saturated sets, we can precisely state their numbers of vertices and simplicial diameter.

\begin{proposition}\label{prop:f_0-n}
    Let $[P_S]$ be the polytope corresponding to a non-empty balanced saturated set $S\subset\Lambda^*$. From the unique decomposition $S = S_{\le q} \cup R$ with $R$ a balanced set of vectors of norm $q+1$ we obtain the number of vertices and simplicial diameter of $[P_S]$ as
    \[
        f_0([P_S]) = |S| = 3\sum_{l=1}^q \varphi(l) + |R|
        \quad\text{and}\quad
        n([P_S]) = \frac{1}{3} \sum_{v\in S} \norm{v} = \sum_{l=1}^q l\,\varphi(l) + \frac{q+1}{3}\,|R|. \qed
    \]
\end{proposition}

Of particular interest are the polytopes corresponding to the saturated sets with $R=\varnothing$, as they are not only vertex-maximal, but uniquely so.

\begin{corollary}
    If $n=\sum_{l=1}^q l\,\varphi(l)$, the only $n$-vertex-maximal polytope is $[P_{S_{\le q}}]$.
\end{corollary}

\begin{proof}
    First note that $[P_{S_{\le q}}]$ is $n$-vertex-maximal by \Cref{cor:saturated-maximal} and \Cref{prop:f_0-n}. From \Cref{lem:saturated} together with the fact, that $S_{\le q}$ is the only saturated set of cardinality $n$, we conclude that $[P_{S_{\le q}}]$ is the only $n$-vertex-maximal polytope.
\end{proof}

We finish this section with a theorem capturing the numerical results we obtain using our construction of vertex-maximal polytopes using saturated sets.

\begin{theorem} \label{thm:sequence-maximal}
    There exists a sequence of vertex-maximal polytopes $([Q_k])_{k\in\N}$ in $\Pc$ with
    \begin{equation*}
        f_0([Q_k]) = 3k = 3\sum_{l=1}^q \varphi(l) + 3r
        \qquad\text{and}\qquad
        n([Q_k]) = \sum_{l=1}^q l\,\varphi(l) + r\,(q+1),
    \end{equation*}
    where $r,q\in\N$ are the unique non-negative integers such that $k=\sum_{l=1}^q \varphi(l)+r$ and $r<\varphi(q+1)$. \qed
\end{theorem}

\begin{proof}
    Given $k=\sum_{l=1}^q\varphi(l)+r$ let $[Q_k]$ be the polytope corresponding to the balanced saturated set $S_{\le q} \cup R$ where $R$ is any balanced set of $3r$ primitive vectors of norm $q+1$. This can always be achieved by choosing $r$ primitive vectors in the half-open cone $\widetilde C_1$ and adding their counterparts from $\widetilde C_2$, $\widetilde C_3$ obtained by cyclic permutation of coordinates with respect to $b_1, b_2, b_3$. Any sequence obtained this way satisfies the desired properties by \Cref{cor:saturated-maximal} and \Cref{prop:f_0-n}.
\end{proof}

    \section{Intermediate Cases}\label{sec-intermediate}

Not for every $n\in\N$ is there a balanced saturated set $S$ such that $[P_S]$ is $n$-vertex-maximal. The reason is that among the saturated sets in \Cref{lem:saturated} there might not be a balanced one. This happens, for example, when $n=\sum_{l=1}^q l\,\varphi(l)+1$, since in this case $R$ has to consist of a single vector and can not possibly be balanced.
However, we can use \Cref{thm:sequence-maximal} to obtain bounds on $A(n)$ and furthermore observe certain properties of $n$-vertex-maximal polytopes for all $n\in\N$.

\begin{proposition} \label{prop:bounds}
    Given $n\in\N$ let $q$ be maximal such that $\sum_{l=1}^q l\,\varphi(l)\le n$ and let $r<\varphi(q+1)$ be maximal such that $\sum_{l=1}^q l\,\varphi(l)+r\,(q+1)\le n$. In case $n=\sum_{l=1}^q l\,\varphi(l)+r\,(q+1)$, we have
    \begin{equation*}
        A(n) = 3\sum_{l=1}^q \varphi(l) + 3r.
    \end{equation*}
    Otherwise, we obtain bounds
    \begin{equation*}
        3\sum_{l=1}^q \varphi(l) + 3r \le A(n) <
        3\sum_{l=1}^q \varphi(l) + 3(r+1).
    \end{equation*}
\end{proposition}

\begin{proof}
    In case of $n=\sum_{l=1}^q l\,\varphi(l)+r\,(q+1)$, the given formula for $A(n)$ is a direct consequence of \Cref{thm:sequence-maximal}. Otherwise, the bounds are obtained from $A(n)$ being weakly increasing and the fact that the polytopes $[P_S]$ for balanced saturated sets $S$ minimize the simplicial diameter among all polytopes of the same number of vertices as shown in \Cref{cor:saturated-maximal}.
\end{proof}

In \Cref{tab:An} we summarize the values for $A(n)$ we could obtain by computation and highlight those implied by \Cref{prop:bounds} with underlines and in case $r=0$ double underlines.

\begin{table}[b]
    \centering
    \caption{The values of $A(n)$ for $n\le 37$ obtained by computation. Values implied by \Cref{prop:bounds} are underlined. Double underlines indicate $r=0$.}
    \begin{tabular}{cc}
        \toprule
        $n$ & $A(n)$ \\ \midrule
        $1$ & \doubleunderline{$3$} \\
        $2$ & $4$ \\
        $3$ & \doubleunderline{$6$} \\
        $4$ & $6$ \\
        $5$ & $8$ \\
        $6$ & \underline{$9$} \\
        $7$ & $10$ \\
        $8$ & $10$ \\
        $9$ & \doubleunderline{$12$} \\
        $10$ & $12$ \\
        \bottomrule
    \end{tabular}
    \quad
    \begin{tabular}{cccc}
        \toprule
        $n$ & $A(n)$ \\ \midrule
        $11$ & $13$ \\
        $12$ & $14$ \\
        $13$ & \underline{$15$} \\
        $14$ & $15$ \\
        $15$ & $16$ \\
        $16$ & $17$ \\
        $17$ & \doubleunderline{$18$} \\
        $18$ & $18$ \\
        $19$ & $19$ \\
        $20$ & $19$ \\
        \bottomrule
    \end{tabular}
    \quad
    \begin{tabular}{cccc}
        \toprule
        $n$ & $A(n)$ \\ \midrule
        $21$ & $20$ \\
        $22$ & \underline{$21$} \\
        $23$ & $21$ \\
        $24$ & $22$ \\
        $25$ & $22$ \\
        $26$ & $23$ \\
        $27$ & \underline{$24$} \\
        $28$ & $24$ \\
        $29$ & $25$ \\
        $30$ & $25$ \\
        \bottomrule
    \end{tabular}
    \quad
    \begin{tabular}{cccc}
        \toprule
        $n$ & $A(n)$ \\ \midrule
        $31$ & $26$ \\
        $32$ & \underline{$27$} \\
        $33$ & $27$ \\
        $34$ & $28$ \\
        $35$ & $28$ \\
        $36$ & $29$ \\
        $37$ & \underline{$30$} \\
        \vphantom{$38$} & \\
        \vphantom{$39$} & \\
        \vphantom{$40$} & \\
        \bottomrule
    \end{tabular}
    \label{tab:An}
\end{table}

\begin{proposition}\label{prop:boundary_edges}
    For every $n>0$, there is an $n$-vertex-maximal polytope $P$ such that $P\cap \partial (n\cdot \Delta)$ consists of one edge of lattice length $1$ on each of the three edges of $n\cdot\Delta$.
\end{proposition}

\begin{proof}
    Let $[Q]$ be any $n$-vertex-maximal polytope with representative $Q\subset n\cdot\Delta$. If $n([Q])<n$, we can instead consider the Minkowski sum $Q+(n-n([Q]))\cdot \Delta$ which is still $n$-vertex-maximal but shares points with each of the three edges of $n\cdot\Delta$ so that $n([Q])=n$.
    
    Now assume that $Q\cap\conv\{n\,a_1,n\,a_2\}=\{v_0\}$ is a single point. We can modify $Q$ in the following way to obtain a polytope $Q'$ with an edge of lattice length $1$ on $\conv\{n\,a_1,n\,a_2\}$ without decreasing the number of vertices. Let $v_0,v_1,\dots,v_k$ be the vertices of $Q$ on the arc from $\conv\{n\,a_1,n\,a_2\}$ to $\conv\{n\,a_2,n\,a_3\}$ so that only $v_0$ and $v_k$ lie on the boundary of $n\cdot\Delta$. Translating the vertices $v_0,\ldots,v_{k-1}$ by $\overline{e_1}$ as depicted in \Cref{fig:enlarge_along_edge} we set
    \[
        Q'= \conv(Q,v_0+\overline{e_1},\ldots,v_{k-1}+\overline{e_1}).
    \]
    Note that none of the new vertices lie outside $n\cdot\Delta$ by our choice of $v_0$ and $v_k$. Furthermore, $Q'$ gained the $k$ vertices $v_0+\overline{e_1},\ldots,v_{k-1}+\overline{e_1}$ that $Q$ does not have, and it lost the vertices $v_1,\ldots,v_{k-1}$ and possibly $v_k$ that $Q$ had. It thus has at least as many vertices as $Q$ and is hence still $n$-vertex-maximal.
    
    \begin{figure}
        \begin{tikzpicture}
            \draw[lightgray] (0,0) -- ++(4,0) -- (0,4);
            \draw (1,2) -- (0,4) -- (2,2);
            \draw[red] (0,0) -- (1,1) -- (1,2);
            \draw[blue] (0,0) -- (1,0) -- (2,1) -- (2,2);
            \draw[thick,->] (0.5,0.5) -- (1.5,0.5);
            \draw[thick,->] (1,1.5) -- (2,1.5);
            \draw[thick,->] (0,4) ++(-90:1.4cm) to[out=0,in=-150] ($(0,4)+(-60:1.4cm)$);
            \node[left] at (0,0) {$v_0$};
            \node[left] at (0,4) {$v_k$};
            \node[below] at (1,0) {$v_0+\overline{e_1}$};
            \node[left] at (1,2) {$v_{k-1}$};
            \node[above right,inner sep=0pt] at (2,2) {$v_{k-1}+\overline{e_1}$};
            \drawgrid{4}
        \end{tikzpicture}
        \caption{The construction used in the proof of \Cref{prop:boundary_edges}.}
        \label{fig:enlarge_along_edge}
    \end{figure}
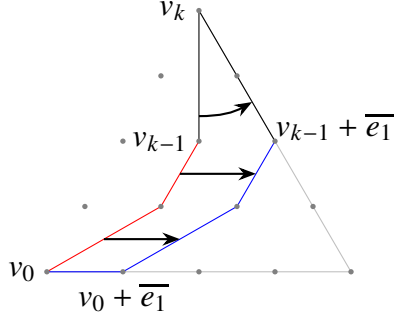
    
    Repeating this procedure for the other two edges of $n\cdot\Delta$ if necessary, we may assume that $Q$ intersects each edge of $n\cdot\Delta$ in an edge. Assume one of these edges has a lattice length $k>1$ and let $v_0$ and $v_1=v_0+k\cdot \overline{e_i}$ denote the vertices of this edge. Simply replacing $v_1$ by $v_0+\overline{e_i}$ then yields a polytope 
    \[
        Q'= \conv(Q^{(0)}\setminus\{v_1\},v_0+\overline{e_i})
    \]
    with the same number of vertices with the considered edge replaced by one of lattice length $1$. Repeating this for all edges of $n\cdot \Delta$ we get a polytope with the desired properties.
\end{proof}

\begin{lemma}\label{lem:boundary_lattice_points}
    Given a lattice polytope $P\subset V$, there is a lattice polytope $P'\subset P$ with the same number of vertices such that all lattice points on the boundary of $P'$ are vertices of $P'$.
\end{lemma}

\begin{proof}
    The statement obviously holds for the special case of $f_0(P)=2$, we thus assume $f_0(P)\ge3$.
    
    \begin{figure}
        \begin{tikzpicture}
            \begin{scope}
                \clip (0,-1) rectangle (3,3.5);
                \fill[shaded] (0,3) -- (0,1) -- (3,0) -- (4,0) -- (4,1) -- cycle;
                \draw (0,3) -- (0,1) -- (1,0) -- (4,0) -- (4,1);
                \node[below left] at (0,1) {$v_{i-1}$};
                \node[below] at (1,0) {$v_i$};
                \node[below] at (3,0) {$x$};
                \node[below] at (4,0) {$v_{i+1}$};
                \draw[thick] (-0.5,1.166667) -- (3.5,-0.166667);
                \node[below] at (1,0.666667) {$L$};
                \node[above right] at (1,1) {$P\cap H^+$};
                \drawgrid{5}
            \end{scope}
        \end{tikzpicture}
        \caption{The construction used in the proof of \Cref{lem:boundary_lattice_points}.}
        \label{fig:cut_edge}
    \end{figure}
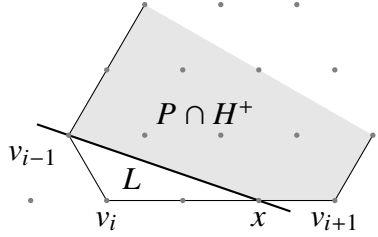

    We shall use the following notation for open and closed line segments:
    \begin{align*}
        (x, y) &:= \{ x + t (y-x) \in V : t \in (0,1) \}\text{ and}\\
        [x, y] &:= \{ x + t (y-x) \in V : t \in [0,1] \}.
    \end{align*}

    Denote the vertices of $P$ by $v_1,\ldots,v_m$ such that $[v_i,v_{i+1}]$ for $i=1,\ldots,m$ are the facets of $P$, where indices are treated modulo $m$. Assume there is a lattice point $x \in (v_i,v_{i+1})$ in the interior of one of these facets. Denote the line through $x$ and $v_{i-1}$ by $L$ as indicated in \Cref{fig:cut_edge}. Note that $v_i\notin L$, as otherwise $\conv(v_{i-1},v_i,v_{i+1})\subset L$, contradicting that $v_{i-1}$, $v_i$ and $v_{i+1}$ are vertices of $P$. We denote the closed half space with boundary $L$ not containing $v_i$ by $H^+$ and set
    \begin{equation*}
        P'=P\cap H^+=\conv(v_1,\ldots,v_{i-1},x,v_{i+1},\ldots,v_m).
    \end{equation*}
    Hence, $P'$ is a lattice polytope with $f_0(P')=f_0(P)$, because for $j\neq i$ we have $v_j\in H^+$, and thus these $v_j$ and $x$ are the vertices of $P'$. Finally, $P'$ contains fewer lattice points than $P$---in particular, $v_i\in P\setminus P'$. Iterating this procedure terminates after finitely many steps, as $P$ only contains finitely many lattice points. We finally obtain a polytope with the same number of vertices as $P$ and no further lattice points on the boundary.
\end{proof}

Applying \Cref{lem:boundary_lattice_points} to our situation, we immediately obtain the following.

\begin{corollary}
    For every $n\in\N$ there is an $n$-vertex-maximal polytope $P$ such that all lattice points on its boundary are vertices. \qed
\end{corollary}

\begin{remark}
    In the proof of \Cref{lem:boundary_lattice_points}, we can choose the orientation of the vertex labeling in each iteration. Thus, after first applying \Cref{prop:boundary_edges}, we can avoid changing the edges of length 1 on the boundary of $n\cdot\Delta$ if there are at least 3 vertices on every arc, yielding a vertex-maximal lattice polytope with both properties.
\end{remark}

    \section{Applications} \label{sec-related}

The asymptotic behavior of $A(n)$ for more general bounding shapes than only the simplex has been studied in \cite{BP}. When applied to the situation above, \cite[Theorem~1.1]{BP} states that
\begin{equation}\label{eq:asymptotic}
    \lim_{n\rightarrow\infty} \frac{A(n)}{n^{2/3}} = \frac{3\AP(\Delta_0)}{(2\pi)^{2/3}}
\end{equation}
where $\AP(\Delta_0)$ denotes the affine perimeter of the so-called ``limit-shape'' of $\Delta$. The latter is detailed to consist of parabolas joined differentiably at points on the boundary of $\Delta$, as depicted in \Cref{fig:limit_shape}. Its affine perimeter is an integral over its curvature and evaluates to $\AP(\Delta_0)=3$.

\begin{figure}
    \newcommand{\n}{17}
    \begin{tikzpicture}[scale=1/3]
        \draw (0,0) -- (\n,0) -- (0,\n) -- cycle;
        \begin{scope}[y={(0cm,1cm)}]
            \draw[edge]
                (60:\n/2) parabola[parabola height=\n*0.217cm] ++(\n/2,0);
            \draw[edge,rotate around={120:(30:0.866*\n*2/3)}]
                (60:\n/2) parabola[parabola height=\n*0.217cm] ++(\n/2,0);
            \draw[edge,rotate around={-120:(30:0.866*\n*2/3)}]
                (60:\n/2) parabola[parabola height=\n*0.217cm] ++(\n/2,0);
        \end{scope}
        \drawgrid{\n}
        \draw[gray]
            (0,9) node[vertex] {} --
            (1,11) node[vertex] {} --
            (2,12) node[vertex] {} --
            (3,12) node[vertex] {} --
            (5,11) node[vertex] {} --
            (8,9) node[vertex] {} --
            (9,8) node[vertex] {} --
            (11,5) node[vertex] {} --
            (12,3) node[vertex] {} --
            (12,2) node[vertex] {} --
            (11,1) node[vertex] {} --
            (9,0) node[vertex] {} --
            (8,0) node[vertex] {} --
            (5,1) node[vertex] {} --
            (3,2) node[vertex] {} --
            (2,3) node[vertex] {} --
            (1,5) node[vertex] {} --
            (0,8) node[vertex] {} --
            cycle;
    \end{tikzpicture}
    \caption{The limit shape consisting of three parabolas inscribed in $17\cdot\Delta$, containing the unique 17-vertex-maximal polytope $P_{S_{\le 4}}$.}
    \label{fig:limit_shape}
\end{figure}
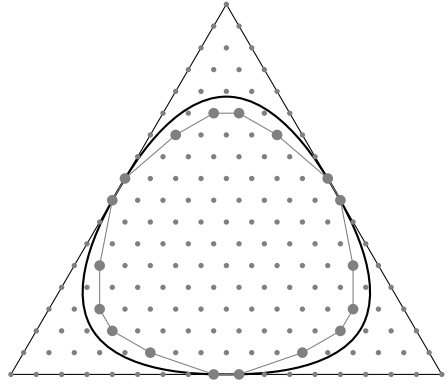

With the help of the sequence of vertex-maximal polytopes $([P_{S_{\le q}}])_{q\in\N}$, one can obtain a combinatorial proof that $AP(\Delta_0)=3$.
We will write $P_q := P_{S_{\le q}}$ for short. Using that
\begin{equation} \label{eq:sumphi}
    \sum_{l=1}^q \varphi(l) = \frac{3\,q^2}{\pi^2} + \Oc(q\log q),
\end{equation}
as shown in \cite{Mertens}, we get
\[ f_0([P_q]) = 3\sum_{l=1}^q \varphi(l) = \frac{9\,q^2}{\pi^2}\ + \Oc(q\log q). \]
For $n([P_q])$, we first rewrite the appearing sum as
\begin{equation*}
    n([P_q]) = \sum_{l=1}^q l\,\varphi(l) = q\sum_{l=1}^q \varphi(l) - \sum_{k=1}^{q-1} \sum_{l=1}^k \varphi(l).
\end{equation*}
Now we can use \cref{eq:sumphi} to obtain
\begin{equation*}
    n([P_q]) = \frac{3\,q^3}{\pi^2} + \Oc(q^2\log q) - \sum_{k=1}^{q-1} \left(\frac{3\,k^2}{\pi^2} + \Oc(k\log k)\right).
\end{equation*}
Finally, using the well known summation formula for summing consecutive squares, we arrive at
\begin{equation*}
    n([P_q]) = 
    \frac{3\,q^3}{\pi^2} + \Oc(q^2\log q) - \left( \frac{q^3}{\pi^2} + \Oc(q^2\log q)\right)
    =
    \frac{2\,q^3}{\pi^2} + \Oc(q^2\log q).
\end{equation*}
We can thus calculate that
\begin{align*}
    \lim_{q\to\infty}\frac{f_0([P_q])^3}{n([P_q])^2}
    &= \frac{9^3}{(2\pi)^2}.
\end{align*}
By taking third roots on both sides, we receive back \cref{eq:asymptotic} with $\AP(\Delta_0)=3$.

\bigskip

Another application of the above results is that we get a bound on the number of rays in a given plane tropical curve in terms of the degree. This number was called $\lambda(d,3)$ in \cite[Corollary~11]{JS} and the first values of this are given in the second row of \cite[Table~1]{JS}. We can now complete an infinite number of entries in this row precisely, and give bounds for the others using \Cref{prop:bounds}.

For the following we use the setting and notation of \cite{GSW}. 
A \emph{plane tropical curve} (in the constant coefficient case) is a pure $1$-dimensional fan $C$ in $V=\Lambda\otimes \R$ together with a multiplicity $\mu(\rho)$ for each $1$-dimensional cone such that if $C$ is pointed we have
\[ \sum_{\rho\in C} \mu(\rho) u_{\rho}=0,\]
where $u_{\rho}$ is the unique primitive generator of $\rho$ with respect to $\Lambda$. If it is pointed, a tropical curve can be expressed by the set $P(C)=\{\mu(\rho_1)u_1,\ldots,\mu(\rho_k)u_k\}\subset \R^{n+1}$, where we choose $u_{i}\in\Z^{3}$ to be the unique representative of $u_{\rho_i}\in \Lambda\otimes \R$ such that the minimum of the coordinates of $u_i$ is $0$.

The \emph{degree} $\deg(C)$ of a tropical curve $C$ in $V$ is the intersection product $C\cdot L$, where $L$ is a generic tropical hyperplane in $V$. Recall that $\deg(C)=d$ if $\sum_{v\in P(C)}v=d\cdot\one$.

We can now give bounds on the number of rays of a plane tropical curve $C$ in terms of its degree (independently of the question of whether $C$ is realizable in a given algebraic plane or not). These bounds are sharp in an infinite number of cases.

\begin{proposition}
Let $C$ be a tropical curve of degree $d$ in $V$. Then $C$ has at most $A(d)$ rays, where $A(d)$ can be computed or bounded as in \Cref{prop:bounds}.
\end{proposition}

\begin{proof}
Let $\Newt(C)\subset V\cong\R^3/\R\one$ be the Newton polytope of $C$ as defined in \cite[Definition~4.13]{GSW}. By \cite[Lemma~4.14]{GSW} we have $\Newt(C)\subset d\cdot\Delta$, and $\Newt(C)$ touches all three sides of $d\cdot\Delta$. By definition of $A(d)$, the polytope $\Newt(C)$ has at most $A(d)$ vertices, and hence, at most $A(d)$ edges. As the edges of $\Newt(C)$ correspond to the rays of $C$, the claim follows.
\end{proof}

Note that this result, however, does not immediately generalize to constant-coefficient tropical curves in general, even if the curve is contained in a matroid fan of a loop-free matroid of rank $3$ (and could then also be called \emph{plane}, as for example in \cite{GSW}):

\begin{example}
Let $n\in\N$, $n\geq1$ and denote by $e_i$ the $i$-th standard basis vector in $\R^{n}$.
\begin{enumerate}
\item If $P(C)=\{\overline{e_1},\ldots,\overline{e_n}\}\subset \R^{n}/\R\one$, then the tropical curve $C$ has $n$ rays and is contained in every matroid fan of a loop-free rank $3$ matroid on $\{1,\ldots,n\}$. However, we have $\deg(C)=1$, so $A(1)=3$ is not an upper bound for the number of rays when $n\geq3$.

\item If $P(C)=\{\overline{e_i} : i=1,\ldots,n\} \cup \{\overline{e_i}+\overline{e_j} : i,j=1,\ldots,n,\ i\neq j\}\subset \R^{n}/\R\one$, then $C$ is a curve contained in the matroid fan $L^{n-1}_2$ and it has $$n+\binom{n}{2}=\frac{n(n+1)}{2}$$ rays. Moreover, $\deg(C)=n$. But $A(n)<\frac{n(n+1)}{2}$ for every $n\geq4$.
\end{enumerate}
\end{example}

    \section*{Acknowledgments}

We would like to thank Masahiko Yoshinaga and Raman Sanyal for fruitful discussions on the subject of this note.

    \bibliographystyle{amsalpha}
    \bibliography{references}
\end{document}